\documentclass{amsart}
\usepackage{yhmath}  
\usepackage{amsmath,amsthm}
\usepackage{marvosym}
\usepackage{graphicx}
\usepackage{color}
\usepackage{epsfig}
\usepackage{enumerate}
\usepackage{cite}
\swapnumbers
\usepackage[active]{srcltx}

\theoremstyle{plain}
\newtheorem{theorem}{Theorem} 

\newtheorem{cor}[theorem]{Corollary}
\newtheorem{conj}[theorem]{Conjecture}

\newtheorem{lemma}[theorem]{Lemma}

\newtheorem{prob}[theorem]{Problem}

\newtheorem*{claim*}{Claim}

\theoremstyle{definition}

\newtheorem{remark}[theorem]{Remark}

\newcommand{\comment}[1]{}

\newcommand{\bdry}{\ensuremath{\partial}}

\newcommand{\nbhd}{\ensuremath{\mathcal{N}}}

\newcommand{\Q}{\ensuremath{\mathbb{Q}}}

\newcommand{\Z}{\ensuremath{\mathbb{Z}}}

\newcommand{\RP}{\ensuremath{\mathbb{RP}}}

\renewcommand{\b}{\ensuremath{\mathfrak{b}}}

\newcommand{\PSL}{\ensuremath{\mathrm{PSL}}}

\newcommand{\mobius}{M\"{o}bius }

\newcommand{\mat}[4]{\left(\begin{smallmatrix} #1 & #2 \\ #3 & #4 \end{smallmatrix}\right)}

\newcommand{\col}[2]{\left(\begin{smallmatrix} #1\\ #2 \end{smallmatrix}\right)}

\title{A Cabling Conjecture for Knots in Lens Spaces}

\author{Kenneth L.\ Baker}
\address{Department of Mathematics\\University of Miami\\ Coral Gables, FL 33146 \\ USA}
\email{k.baker@math.miami.edu}

\dedicatory{Dedicated to the 70th birthday of Professor Fico Gonz\'alez Acu\~na.}
\begin{document}
\maketitle
\begin{abstract}
Closed $3$--string braids admit many bandings to two-bridge links.  By way of the Montesinos Trick, this allows us to construct infinite families of knots in the connected sum of lens spaces $L(r,1) \# L(s,1)$ that admit a surgery to a lens space for all pairs of integers $(r,s)$ except $(0,0)$. These knots are typically hyperbolic. We also demonstrate that the previously known two families of examples of hyperbolic knots in non-prime manifolds with lens space surgeries of Eudave-Mu\~noz--Wu and Kang all fit this construction.  As such, we propose a generalization of the Cabling Conjecture of Gonzales-Acu\~na--Short for knots in lens spaces.  
\end{abstract}

\section{The main results, context, and conjectures}
Banding the braid closure $L=\widehat{\beta}$ of a closed $3$--string braid $\beta$ along a ``level'' framed arc $a$ produces a two-bridge link $L'=\wideparen{\beta}$ that is a plait closure of $\beta$ and a dual framed arc $a'$ as shown in Figure~\ref{fig:3sbto2br}.  (The arcs here are framed by the plane of the page.)  Repeating this with conjugates of $\beta$ often gives many bandings from $L$ to two-bridge links.

\begin{figure}[h!]
\centering
\includegraphics[height=1in]{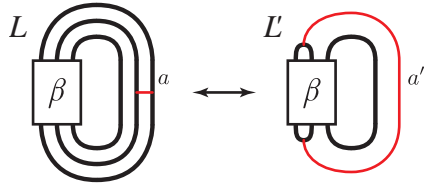}
\caption{A banding between the braid closure $\widehat{\beta}$ and the plait closure $\wideparen{\beta}$ of a $3$--string braid $\beta$.}
\label{fig:3sbto2br}
\end{figure}

Let $B_3$ be the $3$--string braid group with standard generators $\sigma_1, \sigma_2$, and let $\rho\colon B_3 \to \PSL_2\Z$ be defined by $\rho(\sigma_1) =\pm \mat{1}{0}{-1}{1}$ and $\rho(\sigma_2)=\pm \mat{1}{1}{0}{1}$. (See Section~\ref{sec:braids}.)
Let us say a {\em slope} is an ordered pair of coprime integers $(p,q)$ taken up to overall sign.  Each slope $(p,q)$ then corresponds to an extended rational number $q/p \in \widehat{\Q} = \Q \cup \{\infty\}$ which we will also call a slope.
\begin{theorem}\label{thm:closedbraidbanding}
Fix $\beta \in B_3$, and let $\rho(\beta) = \mat{a}{b}{c}{d}$.  Then for every slope $(p,q)$ there is a framed arc $a_\beta^{p,q}$ on the link $\widehat{\beta}$ giving a banding to the two-bridge link
$\b(
-b p^2  - (a - d) p q +c q^2, 
 b p p' + (a - d) p q'  -c q q' + a 
 )$ 
where $p',q' \in \Z $ such that $qp'-pq'=1$.
\end{theorem}

Using the Birman-Menasco's classification of non-prime links with braid index $3$ \cite{birmanmenasco} and the Montesinos Trick \cite{montesinos}, these types of bandings allow us to construct many examples of hyperbolic knots in lens spaces with surgery to non-prime $3$--manifolds.  

\begin{theorem}\label{thm:main}
For each pair of integers $(r,s)$ and slope $(p,q)$ there is a knot $K_{r,s}^{p,q}$ in $L(r,1) \# L(s,1)$ that admits an integral surgery to the lens space
$L(s p^2   - r s  p q + r q^2, s p p'  - r s p q'  + r q q' -1)$
for $p',q' \in \Z$ such that $q p' - p q' = 1$.

For each $(r,s) \neq (0,0)$, infinitely many of these knots are hyperbolic and distinct.

For each $(r,s)$ with both  $|r|,|s| \not \in \{0,2,3\}$,  the set of volumes of the hyperbolic knots $K_{r,s}^{p,q}$ is unbounded.
\end{theorem}

\begin{remark} \
\begin{enumerate}
\item The reasons for excluding both $|r|$ and $|s|$ from being $2$ or $3$ in the last statement of the theorem is an artifact of the proof.  One expects it to hold true for all pairs of integers $(r,s) \neq (0,0)$.

\item
When $(r,s) = (0,0)$, the knots $K_{0,0}^{p,q} \subset S^1 \times S^2 \# S^1 \times S^2$ are all homeomorphic and non-hyperbolic.  They are surgery dual to the trivial knot in $S^1 \times S^2$.

\item
The knots $K_{r,s}^{p,q}$ are analogous to Berge's doubly primitive knots of families VII and VIII, the knots lying in the fiber of a trefoil or the figure eight knot, \cite{berge} (see also \cite{bakergofk,yamadagofk}).  Indeed, they comprise exactly those families when $r=s=\pm1$ and $r=-s=\pm1$ respectively.

\item 
If, say, $s=\pm1$, then the knots $K_{r,\pm1}^{p,q}$ are in $L(r,1)$ and have non-trivial lens space surgeries.
In particular, when $\{r,s\}=\{0,\pm1\}$  these knots form the family  of knots in $S^1 \times S^2$ called {\sc gofk} by Baker--Buck--Lecuona \cite{bbl} and called $A_{m,n}$ by Kadokami--Yamada \cite{kadokami-yamada}.

\item
When $s\neq \pm1$ and $r\neq \pm1$, the manifold $L(r,1)\#L(s,1)$ is a non-trivial connected sum.  According to Theorem~\ref{thm:main}, if furthermore $|r|,|s|>1$ there are infinitely many hyperbolic knots $K_{r,s}^{p,q}$ in the non-trivial connected sum of lens spaces $L(r,1)\#L(s,1)$ all give surgeries to lens spaces. Viewing this dually gives the following corollary.
\end{enumerate}
\end{remark}

\begin{cor}\label{cor:nonprime}
For each pair of integers $(r,s)$ with $|r|,|s|>1$, there exists infinitely many lens spaces containing hyperbolic knots that admit a surgery to the non-trivial connected sum $L(r,1) \#L(s,1)$.
\end{cor}

Boyer--Zhang show that after $6$--surgery on one component of the Whitehead link, the other component is a hyperbolic knot in the lens space which admits a surgery to the non-prime manifold $L(-3,1)\#L(-2,1)$ \cite[Example 7.8]{BoyerZhang-CSseminorms}.  To our knowledge, the only other previously known examples of hyperbolic knots in lens spaces with surgeries to a non-prime manifold are the family of Eudave-Mu\~noz--Wu \cite{emwu} (which contains the Boyer--Zhang example) and the family of Kang \cite{kang}.   
  It turns out that both of these families arise from our construction as we prove in Section~\ref{sec:emwukang}.

\begin{theorem}\label{thm:subsume}
The manifolds $M_p$ of Eudave-Mu\~noz--Wu \cite{emwu} are exteriors of the knots $K^{p-1,p}_{-3,-2}$ in the manifold $L(-3,1)\#L(-2,1)$.

The manifolds $M_{p,q}$ of Kang \cite{kang} are exteriors of the the knots $K^{2,3}_{p,q-2}$ in the manifold $L(p,1)\#L(q-2,1)$.
\end{theorem}

\begin{remark}\
\begin{enumerate}
\item The catalog of exceptional fillings of the Magic $3$--Manifold by Martelli-Petronio \cite{mp} and its extension to the Minimally Twisted $5$--Chain Link by Martelli-Petronio-Roukema \cite{mpr}  appear to offer no further examples of hyperbolic knots in lens spaces with non-prime surgeries.  
\item Note that knots in lens spaces may be presented as surgery on an unknotted component of two-component links.  Thus one may try  examining families of two-components links with at least one unknotted component that have non-prime surgeries to find examples of hyperbolic knots in lens spaces with non-prime surgeries.   The work of Goda-Hayashi-Song does this for two-bridge links \cite{ghs}, and the only hyperbolic knot they obtain is the one found by Boyer-Zhang as surgery on the Whitehead link.  Their other knots are all torus knots or cables of torus knots.
\end{enumerate}
\end{remark}

\subsection{The Lens Space Cabling Conjecture}

Contrast Corollary~\ref{cor:nonprime} with the Cabling Conjecture of Gonzales-Acu\~na--Short \cite{cablingconj} and Greene's confirmation of it in the special case of connected sums of lens spaces \cite{greene}.

\begin{conj}[The Cabling Conjecture, Gonzales-Acu\~na--Short \cite{cablingconj}]
If a knot in $S^3$ admits a surgery to a non-prime $3$--manifold, then the knot is either a torus knot or a cabled knot.
\end{conj}

\begin{theorem}[Greene \cite{greene}]
If surgery on a knot in $S^3$ produces a non-trivial connected sum of lens spaces, then the knot is either a torus knot or a cable of a torus knot.
\end{theorem}

The Cabling Conjecture fails, however, when lens spaces replace $S^3$ in its statement.  This failure was initially demonstrated by the example of Boyer-Zhang.
Yet since our knots of Theorem~\ref{thm:main} encompass all known examples of the failure of the Cabling Conjecture for lens spaces, we propose the following generalization.
\begin{conj}[The Lens Space Cabling Conjecture]
Assume a knot $K$ in a lens space $L$ admits a surgery to a non-prime $3$--manifold $Y$.  If $K$ is hyperbolic, then $Y = L(r,1)\#L(s,1)$, for some slope $(p,q)$ the surgery dual is the knot $K_{r,s}^{p,q} \subset Y$, and $L$ is the lens space of Theorem~\ref{thm:main}.  Otherwise either $K$ is a torus knot, a Klein bottle knot, or a cabled knot and the surgery is along the boundary slope of an essential annulus in the exterior of $K$, or $K$ is contained in a ball.
\end{conj}

In a lens space let us say a {\em torus knot} is a knot that embeds in the Heegaard torus while a {\em Klein bottle knot} is a knot that embeds in a Klein bottle with annular complement; see Section~\ref{sec:nonhyp}.

\begin{remark}
Due to the manifold $S^1 \times S^2$, we specify that $Y$ is non-prime rather than reducible in the conjecture.  There are many more hyperbolic knots in lens spaces that admit surgeries to $S^1 \times S^2$ than discussed in this article.  See \cite{bbl} for our conjectural classification of knots in lens spaces with surgeries to $S^1 \times S^2$.
\end{remark}

One may show that this encompasses a conjecture of Mauricio.
\begin{conj}[Mauricio \cite{maurothesis}]
If $K$ is a knot in $L(p,q) \neq S^3, \RP^3$ with irreducible exterior, then surgery on $K$ never produces $L(p,q) \# L(t,1)$.
\end{conj}

 Boyer-Zhang establish a foundation for the Lens Space Cabling Conjecture.
 \begin{theorem}[Boyer-Zhang, Corollary 1.4  \cite{BoyerZhang-CSseminorms}]\label{thm:bz}
 If a knot in a lens space admits a Dehn surgery to a reducible manifold, then either the surgery is integral or the exterior of the knot is Seifert fibered or reducible.
 \end{theorem}
 
This allows us to reduce the conjecture to integral surgeries on hyperbolic knots.
\begin{theorem}\label{thm:nonhyp}
Non-hyperbolic knots in lens spaces satisfy the Lens Space Cabling Conjecture.
\end{theorem}
\begin{proof}
If $K$ is a non-hyperbolic knot in a lens space that is not contained in a ball, then the exterior of $K$ is irreducible and either Seifert fibered or toroidal by Thurston's Geometrization for Haken manifolds \cite{thurston-geometrization}. 
Lemma~\ref{lem:seifertfibered} addresses the knots in lens spaces with Seifert fibered exteriors.
 Lemma~\ref{lem:satellite} addresses the knots in lens spaces with toroidal exteriors.
\end{proof}

\subsection{Doubly primitive knots and simple knots}
To find hyperbolic knots in lens spaces with surgeries giving non-prime manifolds, we searched for hyperbolic doubly primitive knots in non-prime manifolds. 
A {\em doubly primitive} presentation of a knot $K$ in a manifold $Y$ is an embedding of $K$ in a genus $2$ Heegaard surface $F$ for $Y$ such that the handlebodies bounded by $F$ each have a compressing disk whose boundary $K$ transversally intersects once.  Surgery on $K$ along slope induced by $F$, the doubly primitive surgery, then produces a lens space.  Berge introduced this property as a means of creating knots in $S^3$ with lens space surgeries  \cite{berge}. Saito provides a survey of that work in the appendix of his article about the surgery duals to doubly primitive knots \cite{saito}. 

It is a simple observation now that a non-prime manifold may contain a doubly primitive knot only if it has Heegaard genus $2$ and is therefore a connected sum of lens spaces other than $S^3$.    Based on our experience with the hyperbolicity of knots that embed in the fibers of genus one fibered knots \cite{bakergofk}, the knots $K^{p,q}_{r,s}$ were a natural class to derive.  They each embed in the fiber of a genus one fibered knot in $L(r,1)\#L(s,1)$.   Because of the following lemma, these are the only doubly primitive knots in non-prime manifolds of this type.

\begin{lemma}
The only non-prime manifolds contain a genus one fibered knot are the manifolds $L(r,1) \# L(s,1)$ where $r,s \neq \pm1$.
These genus one fibered knots are the plumbings of the $r$--Hopf band in $L(r,1)$ and the $s$--Hopf band in $L(s,1)$.
\end{lemma}
The {\em $r$--Hopf band} is the fibered link in $L(r,1)$ whose fiber is an annulus and the monodromy is $r$--Dehn twists along the core curve; see \cite{bjk}.

\begin{proof}
The plumbings of the $r$--Hopf band in $L(r,1)$ and the $s$--Hopf band in $L(s,1)$ clearly give genus one fibered knots in  $L(r,1) \# L(s,1)$.
Towards the other direction, note that every genus one fibered knot is the lift of the braid axis in the double branched cover of a link presented as a closed $3$--braid \cite{bakercgofk}.
(A once-punctured torus bundle $X$ is the double branched cover of a $3$--braid in a solid torus.  The induced involution on $X$ acts freely on $\bdry X$ and takes curves in $\bdry X$ that intersect each fiber just once to disjoint curves.  Hence the involution extends freely across a filling of $X$ along such a slope.  In the quotient, the core of the filling solid torus descends to a braid axis for an unknotted embedding into $S^3$ of the solid torus containing the $3$--braid.)
Since Birman-Menacso show that the only non-prime links with braid index $3$ are the links $\b(r,1)\#\b(s,1)$ (where both $r,s \neq \pm1$), and their only presentations as closed $3$--braids are the obvious ones \cite{birmanmenasco}, there can be no other genus one fibered knots in non-prime $3$--manifolds.
\end{proof}

Our attempts to find other hyperbolic doubly primitive knots in connected sums of lens spaces generally failed.  More specifically, our attempts to generalize Berge's sporadic knots \cite{berge} and the duals of Tange's knots \cite{tange} to non-prime manifolds usually resulted in knots that were surgery dual to cabled knots if their exteriors weren't Seifert fibered.  This could be explored further.

\begin{conj}
If $K$ is a hyperbolic doubly primitive knot in a non-prime manifold, then $K=K^{p,q}_{r,s}$ for some slope $(p,q)$ and integers $r,s \neq \pm1$.
\end{conj}

The dual knot to the doubly primitive surgery on a doubly primitive knot is necessarily a $(1,1)$--knot in the resulting lens space, and integral surgery on any $(1,1)$--knot naturally produces a doubly primitive knot.

\begin{prob}\label{prob:doublyprim}
Classify the doubly primitive knots in non-prime manifolds.
Equivalently, classify the $(1,1)$--knots in lens spaces with non-prime surgeries.
\end{prob}

Among $(1,1)$--knots is a special class of knots, often called {\em simple knots}.
In the lens space $L(p,q)$ with $|p|>2$, choose a compressing disk to each side of the Heegaard torus so that their boundaries intersect $p$ times.
Labeling the intersections around one disk as $0, 1, 2, \dots, p-1$ so that they appear as $0, q, 2q, \dots, (p-1)q$ (reduced mod $p$) around the other, the {\em simple knot} $K(p,q,k)$ in $L(p,q)$ for $k \in \{1, 2, \dots, p-1\}$ is the knot that may be expressed as the union of the arc in each of these disks connecting $0$ and $k$.  Observe that $[K(p,q,1)]$ is a generator of $H_1(L(p,q))$ and $[K(p,q,k)] = k[K(p,q,1)]$.  Extend this by defining $K(p,q,k)$ to be the trivial knot if $k=0$ or $|p|\leq1$.   Simple knots are also known as {\em grid number one knots}  and are fundamental in the study of knots in lens spaces, \cite{rasmussen, greene, saito, tange, bbl, bgh, heddensimple} to name a few.

\begin{lemma}\label{lem:simple}\
\begin{enumerate}
\item A $(1,1)$--knot in a lens space with an integral surgery to $S^1 \times S^2 \# L(t,u)$ is a simple knot.
\item The dual to a doubly primitive knot in $S^1\times S^2 \# L(t,u)$ is a simple knot in the corresponding lens space.
\end{enumerate}
\end{lemma}

\begin{proof}
The proof exactly follows the proof of Theorem~1.8 of \cite{bbl}, replacing $S^1 \times S^2$ with $S^1 \times S^2 \# L(t,u)$.
\end{proof}

\begin{remark}
The surgery dual to $K^{p,q}_{r,s}$ is a simple knot if either $r=0$ or $s=0$ by Lemma~\ref{lem:simple} or if $|r|=|s|=\pm1$ by Berge \cite{berge}.  However, the duals to our knots $K^{p,q}_{r,s}$ in general are not necessarily simple.  Indeed, Boyer--Zhang's example is a null-homologous knot with genus $1$, but it is not simple since null-homologous simple knots have genus $0$. 
\end{remark}

In light of Lemma~\ref{lem:simple} and that the knot Floer homology of simple knots is simple, the following special case of Problem~\ref{prob:doublyprim} should be more tractable.
\begin{prob}
Classify the simple knots in lens spaces with an integral surgery to $S^1 \times S^2 \# L(t,u)$.
\end{prob}


\section{The basics}

\subsection{Banding and framed surgery}
Let $I$ denote the closed interval $[-1,1]$.
For a link  $L$ in $S^3$.  Let us say an {\em arc} $a$ on $L$ is (the image of) an embedding $a \colon I \to S^3$ such that $a \cap L = \bdry a$.  A {\em framing} of the arc $a$ is (the image of) an embedding $f \colon I \times I \to S^3$ such that $f \vert I \times \{0\} = a$ and $f \cap L = \bdry a \times I$.  A framed arc is also known as a {\em band}.  Framed arcs $a_1$ and $a_2$ on $L$ are {\em isotopic} if there is an isotopy from $a_1$ to $a_2$ through framed arcs on $L$.  They are {\em weakly isotopic} if there is an isotopy between the pairs $(L,a_1)$ and $(L,a_2)$.

A regular neighborhood of a framed arc $a$ on $L$ is homeomorphic to the left of Figure~\ref{fig:banding}.  The right of that figure shows the result of {\em banding} $L$ along the framed arc $a$, producing a link $L'$ with a dual framed arc $a'$.  The figure explicitly shows the framing of each of these arcs.  Such a transformation is also called an integral rational tangle replacement.

\begin{figure}[h!]
\centering
\includegraphics[height=1in]{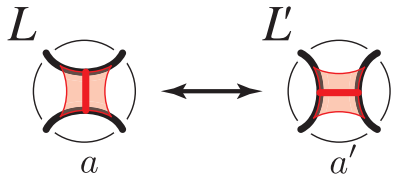}
\caption{ Banding between links $L$ and $L'$ along dual framed arcs $a$ and $a'$.}
\label{fig:banding}
\end{figure}

Let $\Sigma(L)$ denote the double cover of $S^3$ branched along $L$.  In this cover, the framed arc $a$ lifts to a framed knot $K_a$. By the Montesinos Trick \cite{montesinos}, the banding along $a$ lifts to surgery on the framed knot $K_a$ which produces $\Sigma(L')$ and dual framed knot $K_a' = K_{a'}$.


\subsection{Two-bridge links and lens spaces}

Following Kirby-Melvin \cite{kirbymelvin}, given a sequence of integers ${\bf a} = (a_1, \dots, a_n)$ define the continued fraction expansion
\[
 [{\bf a}] = [a_1, \dots, a_n] = -\cfrac{1}{a_1-\cfrac{1}{a_2 - \cfrac{1}{ \ddots -\frac{1}{a_n}}}}
\]
and associated matrix $A_{\bf a} = S T^{a_1} S T^{a_2} S \dots T^{a_n}S \in \PSL_2 \Z$ where $S=\pm\mat{0}{-1}{1}{0}$ and $T =\pm \mat{1}{1}{0}{1}$.   Also set ${\bf a}^{-1} = (-a_n, \dots, -a_1)$.

\begin{lemma}[Lemma 1.9 \cite{kirbymelvin}]
Let ${\bf a} = (a_1, \dots, a_n)$ and $A_{\bf a} = \mat{a}{b}{c}{d}$.
Then
$ a/c =[a_1, \dots, a_n]$, $ b/d = [a_1, \dots, a_{n-1}]$,  $d/c = [a_n, \dots, a_1]$, and $b/a = [a_n, \dots, a_2]$. \qed
\end{lemma}

Recall that for coprime integers $p,q$ the two-bridge link  $\b(p,q)$  may be expressed as one of the left or center of Figure~\ref{fig:twobridgeandchain} for some sequence of integers ${\bf a} = (a_1, \dots, a_n)$ such that $q/p = [{\bf a}]$.    The two-bridge links $\b(p,q)$ and $\b(p',q')$ are isotopic if either $p'=p$ and $q'=q^{\pm1} \mod p$ or $-p'=p$ and $-q' = q^{\pm1} \mod p$.   See e.g.\ \cite{burdezieschang}.  In our diagrams, an oblong rectangle labeled with an integer will denote a sequence of twist as illustrated in Figure~\ref{fig:twistingconventions}.

\begin{figure}[h]
\centering
\includegraphics[height=3in]{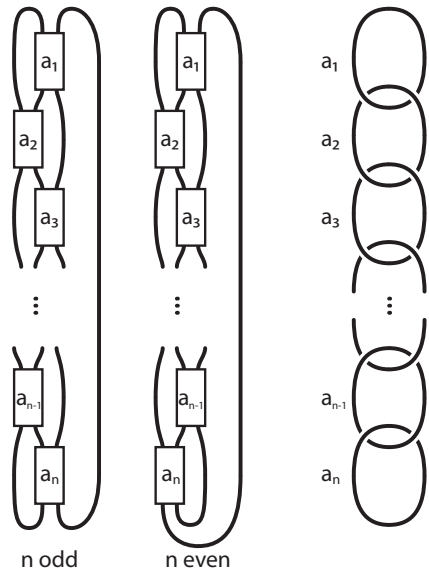}
\caption{For $q/p = [a_1, a_2, \dots, a_n]$, left and center give the two-bridge link $\b(p,q)$ for $n$ odd or even while right gives a surgery description of the lens space $L(p,q)$.}
\label{fig:twobridgeandchain}
\end{figure}

\begin{figure}[h]
\centering
\includegraphics[height=1in]{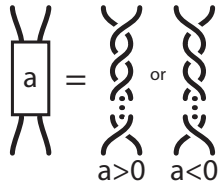}
\caption{Use $|a|$ right-handed or left-handed twists according to the sign of $a$.}
\label{fig:twistingconventions}
\end{figure}

Define the lens space $L(p,q)$ as the result of $-p/q$--Dehn surgery on the unknot.  If $q/p = [a_1, \dots, a_n]$, then $L(p,q)$ also admits a surgery description on the linear chain link shown on the right of Figure~\ref{fig:twobridgeandchain} where the surgery coefficients are the coefficients of the continued fraction, e.g.\ \cite[(2.5)]{kirbymelvin}.
The lens space $L(p,q)$ is the double branched cover of the two-bridge link $\b(p,q)$.


\subsection{Braids}\label{sec:braids}
Let $B_3$ be the three-strand braid group with standard generators $\sigma_1$ and $\sigma_2$ as in Figure~\ref{fig:braidgenerators}.  Let $\rho \colon B_3 \to \PSL_2 \Z$ be the standard homomorphism defined by $\rho(\sigma_1) = \pm T$ and $\rho(\sigma_2) = \pm STS$.  Recall that the kernel of $\rho$ is generated by the full-twist $\Delta^2$ where $\Delta = \sigma_1 \sigma_2 \sigma_1$.  

\begin{figure}[h]
\centering
\includegraphics[width=1in]{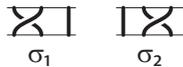}
\caption{The generators $\sigma_1$ and $\sigma_2$ of $B_3$.}
\label{fig:braidgenerators}
\end{figure}

Given $\beta \in B_3$, let $\widehat{\beta}$ denote its {\em braid closure} and let $\wideparen{\beta}$ denote its {\em plait closure} as in the left and right of Figure~\ref{fig:3sbto2br}.
If $\rho(\beta) =\mat{a}{b}{c}{d}$, then $\wideparen{\beta}$ is the two-bridge link $\b(c,a)$. Indeed, there is an integer sequence ${\bf b}$ of odd length such that $\beta = \beta_{\bf b}$, $\rho(\beta) = A_{\bf b}$, and $a/c = [{\bf b}]$.

Let $\phi \colon B_3 \to \hat{\Q}$ be given by $\phi(\beta) = a/c$ if $ \rho(\beta)=\mat{a}{b}{c}{d}$. 
 
\begin{lemma}\label{lem:matequiv}
Given $\beta, \beta' \in B_3$,
if $\phi(\beta) = \phi(\beta')$ then $\beta' = \Delta^{2N} \beta \sigma_1^n$ for some  $n,N \in \Z$.
\end{lemma}

\begin{proof} 
Let us view a number $a/c \in \hat{\Q}$  as $\pm \col{a}{c}$ for coprime integers $a,c$.
In $\PSL_2\Z$, if $A \col{1}{0} = B \col{1}{0}$, then $B^{-1}A \col{1}{0}= \col{1}{0}$ and hence $B^{-1} A = T^n$ for some $n \in\Z$.
 Since $\rho(\sigma_1) = T$ and the element $\Delta^2 \in B_3$ generates the kernel of $\rho$, we have our conclusion.
\end{proof}

\section{The main arguments}
\subsection{Bandings of $3$--string braids}

Given $\beta \in B_3$, write $\beta = \sigma_2^{b_1}\sigma_1^{b_2}\sigma_2^{b_3} \dots \sigma_1^{b_{n-1}} \sigma_2^{b_n}$ for some sequence of integers ${\bf b} = (b_1, \dots, b_n)$ of odd length $n$.  Then $\rho(\beta) = A_{\bf b}$.


\begin{proof}[Proof of Theorem~\ref{thm:closedbraidbanding}]
Given the slope $(p,q)$, choose  $\gamma \in B_3$ such that $\rho(\gamma) = \mat{q}{q'}{p}{p'}$. This then defines an arc $a_\beta^\gamma$ on the link $\widehat{\beta}=\widehat{\gamma^{-1} \beta \gamma}$ framed by the plane of the page as shown in the left of Figure~\ref{fig:basicarc}.  In fact, this arc only depends upon the slope $(p,q)$ and thus defines the arc $a_\beta^{p,q}$ as we observe: By Lemma~\ref{lem:matequiv}, any $\gamma_0,\gamma_1 \in B_3$ with $\phi(\gamma_0) =\phi(\gamma_1)$ satisfies $\gamma_1 = \Delta^{2N} \gamma_0 \sigma_1^n$ for some $n,N \in \Z$.   Since $\Delta^2$ is a full twist in $B_3$ and the arc is between the strands that $\sigma_1$ acts upon, the isotopy of $\widehat{\gamma_0^{-1} \beta \gamma_0}$ to $\widehat{\gamma_1^{-1} \beta \gamma_1}$ through closed braids (fixing $\beta$) takes the arc $a_\beta^{\gamma_0}$ to the arc $a_\beta^{\gamma_1}$. 

\begin{figure}[h]
\centering
\includegraphics[width=2in]{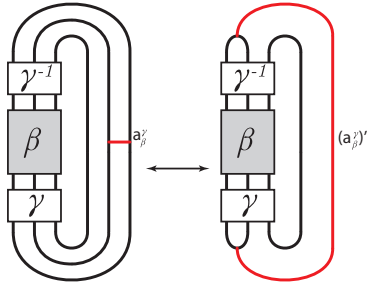}
\caption{The arc $a^\gamma_\beta$ on the braid closure of the $3$--string braid $\beta$.}
\label{fig:basicarc}
\end{figure}

 Figure~\ref{fig:basicarc} shows the banding of $\widehat{\beta}$ along the arc $a_\beta^{p,q}$.  Using $\gamma$ to present $a_\beta^{p,q}$ as  $a_\beta^\gamma$ this produces the two-bridge link $\wideparen{\gamma^{-1} \beta \gamma}$ as shown on the right. 
Using $\rho(\beta) = \mat{a}{b}{c}{d}$ and $\rho(\gamma) = \mat{q}{q'}{p}{p'}$, identifies this link as $\b( -b p^2 - (a - d) p q +c q^2 ,b p p' + (a - d) p q'  -c q q' + a )$.  
\end{proof}

\begin{cor}\label{cor:connectsumbanding}
Fix a pair of integers $(r,s)$.  Then for every slope $(p,q)$  there is a framed arc $a_{r,s}^{p,q}$ on the link $\b(r,1) \# \mathfrak{b}(s,1)$ giving a banding to the link $\b(s p^2 - r s  p q + r q^2, s p p' - r s p q' + r q q'   -1)$  where $p',q' \in \Z$ such that $qp'-pq'=1$.
\end{cor}

\begin{proof}
Let $\beta = \sigma_2^r \sigma_1^s$ so that $\rho(\beta) = \mat{-1}{ -s}{r}{ -1 + r s}$
 and apply Theorem~\ref{thm:main}.
\end{proof}

\subsection{Chain links}

For integers $n\geq2$, let $C_{2n}$ be the minimally twisted (i.e.\ untwisted) chain link of $2n$--components.  Figure~\ref{fig:chainlink} shows two presentations of $C_{10}$. (Thurston calls this link $D_{2n}$  \cite[Chapter 6]{thurston}.)

\begin{figure}[h]
\centering
\includegraphics[height=1in]{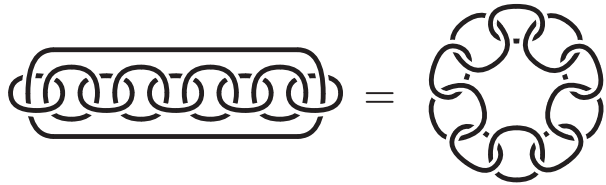}
\caption{The minimally twisted chain link $C_{10}$.}
\label{fig:chainlink}
\end{figure}

\begin{lemma}\label{lem:hypchainlinksurgery}
For integers $n\geq 3$ and $s \neq 0, \pm 2, \pm3$, the result of $s$--surgery on one cusp of the complement of $C_{2n}$ is a hyperbolic manifold.  For $s = \pm2$ or $\pm3$, the result of $s$--surgery on one cusp of the complement of $C_{6}$ is a hyperbolic manifold.
\end{lemma}

\begin{proof}
In \cite[Chapter 6]{thurston}, Thurston shows how to explicitly determine a hyperbolic structure on the complement of $C_{2n}$ when $n\geq 3$.  From this one may determine the similarity classes of cusp shapes, determine the normalized length of the slope $s$, and apply the work of Hodgson-Kerckhoff \cite{HK} to constrain the set of surgery slopes $s$ that possibly do not yield hyperbolic manifolds.  However, as pointed out by Purcell, we can improve upon this by appealing to the $6$--Theorem of Agol \cite{agol6} and Lackenby \cite{lackenby6} and Perelman's Geometrization Theorem \cite{perelman1,perelman2}.

Viewing the link $C_{2n}$ as a chain of alternately horizontal unknots and vertical ``crossing circles'' presents it as an example of a fully augmented link; see \cite{purcell-augmented} for an introduction.  In the Euclidean cross-sections of these links, the meridian and longitude slopes are orthogonal.  Together Lemma~2.6(3) and Corollary~3.9 of \cite{futerpurcell} show that there is a horoball expansion for the cusps of the complement of $C_{2n}$ so that in the cusp corresponding to a horizontal component the meridian has length $2$ and the longitude has length $\lambda >0$.  
When $|s| >3$ the length of the surgery slope will be $\sqrt{4s^2 + \lambda^2} > \sqrt{36+\lambda^2} > 6$.  Then the $6$--Theorem of Agol and Lackenby implies that the resulting filled manifold is hyperbolike which, by Perelman, is equivalent to being hyperbolic.

For $s=\pm1$, the surgery turns $C_{2n}$ into the minimally twisted chain link of $2n-1$ components which Neumann-Reid shows is hyperbolic for $n \geq 3$ \cite{neumannreid}.

For $s=\pm2$ or $\pm3$, SnapPy \cite{snappy} verifies the hyperbolicity of $s$--surgery on one cusp of the complement of $C_6$.  (Actually, we have checked this for $C_8$, $C_{10}$, and $C_{12}$ as well.)
\end{proof}


\subsection{The knots $K^{p,q}_{r,s}$, hyperbolicity, and volumes}

\begin{proof}[Proof of Theorem~\ref{thm:main}]
We follow the proof of Theorem~\ref{thm:closedbraidbanding} and its application to Corollary~\ref{cor:connectsumbanding}, while controlling the specific braids being used.  Through double branched covers this gives a handle on the surgery descriptions on the links $C_{2n}$ of our knots $K^{p,q}_{r,s}$ in $L(r,1) \# L(s,1)$.  Lemma~\ref{lem:hypchainlinksurgery} with Thurston's Hyperbolic Dehn Surgery Theorem then allows our statements about hyperbolicity and volumes.

Given a slope $(p,q)$, choose a sequence ${\bf b} = (b_k, \dots, b_1)$ with $k$ odd such that $q/p = [{\bf b}]=[b_k, \dots, b_1]$.  
Then define the framed arc $a_{r,s}^{p,q}$ framed by the plane of the page on the link $\b(r,1) \# \b(s,1) = \widehat{\beta}$ as on the left of Figure~\ref{fig:definetangle} where $\beta = \sigma_2^r \sigma_1^s$ and $\gamma=\gamma_{\bf b}= \sigma_2^{b_k} \sigma_1^{b_{k-1}} \dots \sigma_2^{b_1} \in B_3$.   As shown, banding along the arc $a^{r,s}_{p,q}$  produces the two-bridge link $\wideparen{\gamma^{-1} \beta \gamma}=\b(t,u)$ where 
\[u/t = [-b_1, \dots, -b_k, 0, r, s, b_k, \dots, b_1] = [-b_1, \dots, -b_k + r, s, b_k, \dots, b_1]\]   
and dual arc $(a^{p,q}_{r,s})'$ on the right of Figure~\ref{fig:definetangle}. 
From Corollary~\ref{cor:connectsumbanding}, this is the two-bridge link
$\b(s p^2 - r s  p q + r q^2, s p p' - r s p q' + r q q'   -1)$.

\begin{figure}[h]
\centering
\includegraphics[width=3in]{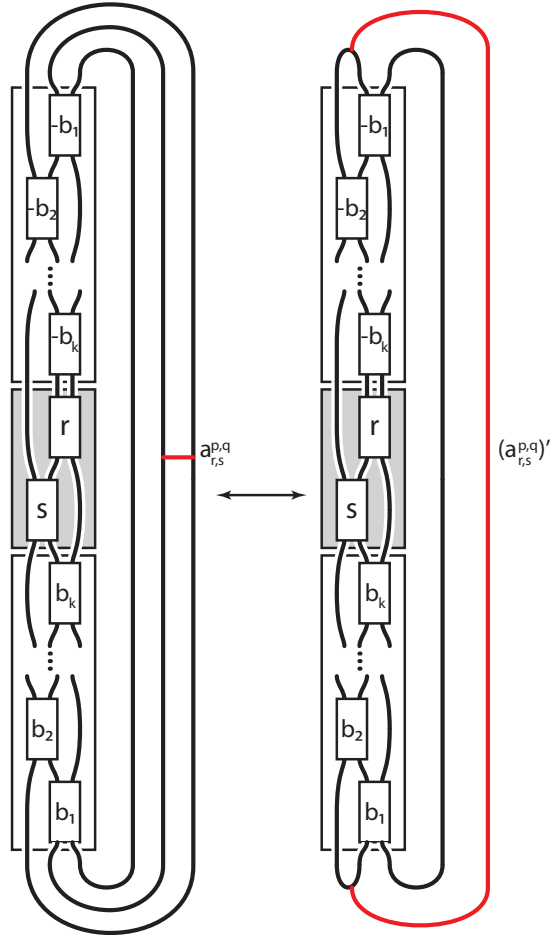}
\caption{The arc $a^{p,q}_{r,s}$ and its banding where $q/p=[b_k, \dots, b_1]$.}
\label{fig:definetangle}
\end{figure}
Taking the double branched cover over the link $\mathfrak{b}(r,1) \# \mathfrak{b}(s,1)$ produces the connected sum $L(r,1) \# L(s,1)$.   We define the knot $K_{r,s}^{p,q}$ to be the lift of the arc $a_{r,s}^{p,q}$ in this cover.  The framing of the arc lifts to a framing of the knot.  Through the Montesinos Trick, the banding along this framed arc lifts to surgery on the knot along this framing.  Hence $K_{r,s}^{p,q}$ admits a surgery to the lens space $L(t,u)$ as claimed.

Observe that any sequence ${\bf b} = (b_k, \dots, b_1)$ with $k$ odd such defines a slope $(p,q)$ where $q/p = [{\bf b}]=[b_k, \dots, b_1]$ and hence a knot $K^{p,q}_{r,s}$.  The exterior of $K_{r,s}^{p,q}$ then admits a surgery description on the link $C_{2(k+1)}$ where the $2k+1$ surgery coefficients are $(b_1, \dots, b_k+r, s, -b_k, \dots, -b_1)$ consecutively. 

Note that by the symmetry of $\b(r,1) \# \b(s,1)$,  for each slope $(p,q)$ there exists a slope $(p',q')$ such that  $a^{p,q}_{r,s}$ and $a^{p',q'}_{s,r}$ are isotopic arcs.  Similarly, for each slope $(p,q)$ there exists a slope $(p',q')$ such that $K^{p,q}_{r,s}$ and $K^{p',q'}_{s,r}$ are isotopic knots in $L(r,1)\#L(s,1)$.  Use this symmetry to assume that $s \neq 0$ if $(r,s) \neq 0$, and further that $|s| \neq 2,3$  when possible.

\smallskip

Assume $s\neq0$.  Then Lemma~\ref{lem:hypchainlinksurgery} shows that $s$--surgery on one cusp of the complement of $C_{2(k+1)}$ gives a hyperbolic manifold as long as either we have $|s|>4$ or $|s|=1$ and $k\geq 2$ or we have $|s| =2$ or $3$ and $k=2$.  By Thurston's Hyperbolic Dehn Surgery Theorem \cite{thurston-geometrization}, there is a constant $\kappa$  such that choosing the sequence ${\bf b}$ with  $|b_i|>\kappa$  for each $i=1, \dots k$ will ensure that $K_{r,s}^{p,q}$ is hyperbolic.  Fixing $k$ odd, if we set $q_{j}/p_{j}$ to have a continued fraction expansion $[j, j, \dots, j]$ of length $k$, then the volumes of $K_{r,s}^{p_{j},q_{j}}$ for $j > \kappa$  will approach the volume of $C_{2(k+1)}$ as $j \to \infty$.  Hence infinitely many of these knots will be hyperbolic and distinguished by their volume.  

Now assume $|s|>4$ or $|s|=1$.   Since the volume of a hyperbolic link is bounded below by the number of its components times the volume $v$ of a regular ideal hyperbolic tetrahedron \cite{adams}, the volume of $C_{2(k+1)}$ is greater than $2(k+1)v$ for $k\geq2$. 
Then for each $k \geq 2$, choose a slope $(p_k, q_k)$ such that $K_{r,s}^{p_k, q_k}$ has volume within $\epsilon$ of the volume of  $C_{2(k+1)}$ for some fixed $\epsilon >0$. Thus the set of volumes of the knots $K_{r,s}^{p_k,q_k}$ for $k \geq 2$ is unbounded.
\end{proof}


\subsection{The knots of Eudave-Mu\~noz--Wu and Kang}\label{sec:emwukang}

\begin{proof}[Proof of Theorem~\ref{thm:subsume}]
Figure~\ref{fig:emwu} shows the isotopy of the tangle $\xi_p$ in Figure 4.1(a) of Eudave-Mu\~noz--Wu \cite{emwu} into the form of the left of Figure~\ref{fig:definetangle} (as well as Figure~\ref{fig:basicarc}) with a neighborhood of the arc removed.  The two small tangles at the bottom left of this figure indicate the fillings of Figure~4.1(a) \cite{emwu} that produce the two-bridge link Figure~4.1(c) and connected sum Figure~4.1(f) respectively.  In the third frame of our figure, an isotopy twists the boundary of the tangle and is compensated in the small filling tangles below it.  At the end of the figure, we recognize our decomposition with $\beta = \sigma_2^{-3} \sigma_1^{-2}$ and $\gamma = \sigma_2^{-1} \sigma_1^{-p}$.  Since $[-1,-p] = \tfrac{p}{p-1}$, the double branched cover of this tangle, Eudave-Mu\~noz--Wu's manifold $M_p$, is the exterior of our knot $K^{p-1,p}_{-3,-2}$.
\begin{figure}[h]
\centering
\includegraphics[width=\textwidth]{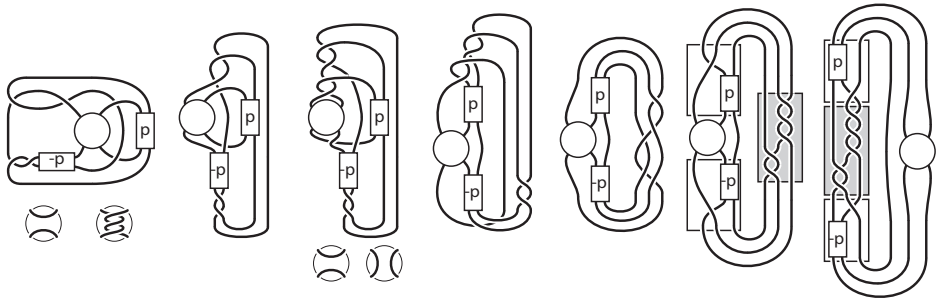}
\caption{An isotopy of the tangle $\xi_p$ of Eudave-Mu\~noz--Wu \cite{emwu}.}
\label{fig:emwu}
\end{figure}

\begin{figure}[h]
\centering
\includegraphics[width=\textwidth]{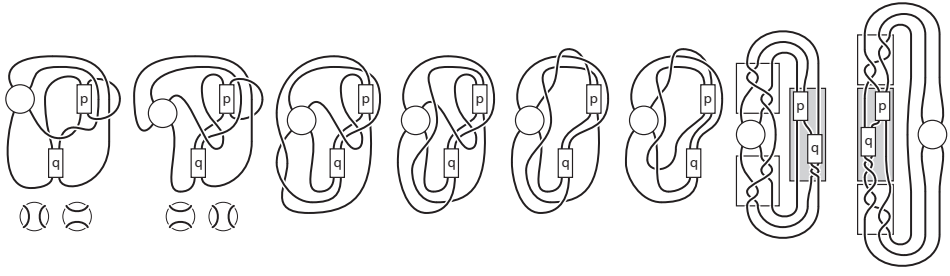}
\caption{An isotopy of the tangle $T_{p,q}$ of Kang \cite{kang}.}
\label{fig:kang}
\end{figure}

Figure~\ref{fig:kang} similarly shows the isotopy of the tangle $T_{p,q}$ in Figure~1 of Kang \cite{kang}. At the end of the figure, we recognize our decomposition with $\beta = \sigma_2^{p} \sigma_1^{q-2}$ and $\gamma = \sigma_2^{0} \sigma_1^{2} \sigma_2^{2}$.  Since $[0,2,2] = \tfrac{3}{2}$, the double branched cover of this tangle, Kang's manifold $M_{p,q}$,  is the exterior of our knot $K^{2,3}_{p,q-2}$.
\end{proof}


\section{Non-hyperbolic knots in lens spaces with non-prime surgeries}\label{sec:nonhyp}

Here we prove the two lemmas cited in the proof Theorem~\ref{thm:nonhyp}, but first let us briefly discuss Klein bottle knots.
As stated in the introduction, a {\em torus knot} is a knot that embeds in the Heegaard torus while a {\em Klein bottle knot} is a knot that embeds in a Klein bottle with annular complement.  We put the extra restriction of annular complement upon Klein bottle knots since the other embedded curves on a Klein bottle in a lens space are all torus knots. Recall that a Klein bottle contains only five distinct isotopy classes of embedded loops: the trivial loop that bounds a disk, the loop that separates the Klein bottle into two \mobius bands, the two cores of these \mobius bands, and the loop whose complement is an annulus.  As Bredon-Wood show, the lens spaces $L(4k,2k-1)$ for $k\in\Z$ are the only ones to contain a Klein bottle and that Klein bottle is unique up to isotopy \cite{bredonwood}.   Such a Klein bottle may be isotoped in the lens space so that the Heegaard torus splits it into two \mobius bands.  The intersection of the Heegaard torus and the Klein bottle is clearly a torus knot.  The two cores of these \mobius bands are cores of the Heegaard solid tori, and hence are torus knots.    The trivial knot that bounds a disk in the Klein bottle may also be regarded as a torus knot.  Our Klein bottle knot may be viewed as the union of a spanning arc in each of the two \mobius bands.  It turns out that this knot is the simple knot $K(4k,2k-1,2k)$, and its exterior Seifert fibers over the \mobius band with one exceptional fiber \cite[Lemma 6.3]{bakerbuck}.  When $|k| \geq 2$ it is not a torus knot \cite[Proposition 5.3]{bakerbuck}.

\begin{lemma}\label{lem:seifertfibered}
Let $K$ be a knot in a lens space $L$ that has Seifert fibered exterior and admits a surgery to a non-prime manifold $Y$. Then either 
\begin{enumerate}
\item $K$ is a torus knot in $L$ and the surgery is along the integral slope induced by the Heegaard torus in which the knot embeds, or
\item $K$ is the Klein bottle knot in the lens space $L(4k,2k-1)$ for some integer $k \neq -1, 0,1$, and the surgery is along the integral slope induced by the Klein bottle in which the knot embeds.
\end{enumerate}
\end{lemma}
\begin{proof}
A knot $K$ in a lens space $L$ with Seifert fibered exterior $X$ is either a torus knot or a Klein bottle knot \cite[Theorem~6.1]{bakerbuck}.  If $K$ has a surgery to a non-prime manifold, then its exterior $X$ Seifert fibers either over the disk with exactly two exceptional fibers or over the \mobius band with exactly one exceptional fiber, and in both cases the surgery slope is isotopic to a regular fiber and integral \cite[Proposition~5.4]{buckmauricio}.  (If $X$ Seifert fibers over the \mobius band with no exceptional fibers, then it also fibers over the disk with two exceptional fibers.) Observe that the Seifert fibrations of these manifolds contain essential annuli that are the union of regular fibers.  Since the slope of the meridian intersects the regular fibers in $\bdry X$ just once, in each case the boundary of these annuli join along $K$ in the meridional filling to produce the respective Heegaard torus or the Klein bottle.
\end{proof}

\begin{lemma}\label{lem:satellite}
Let $K$ be a knot in a lens space whose exterior is irreducible, not Seifert fibered, and toroidal.  If $K$ admits a surgery to a non-prime manifold then $K$ is a cabled knot.
\end{lemma}

\begin{proof}
Let $T$ be an essential torus in the irreducible exterior $X$ of a knot $K$ in a lens space $L$.  Then $T$ is separating in $L$ and bounds a $\bdry$--compressible manifold $M \subset L$ that contains $K$. Also, any sphere in $M$ either bounds a ball or a punctured $L$ in $M$.  Assume $K$ admits a surgery to a non-prime manifold $Y$, and let $K'$ be the knot dual to the surgery.  This surgery also transforms $M$ to a manifold $M' \subset Y'$ still bounded by $T$ and containing the knot $K'$.  Then by Scharlemann \cite{scharlemannreducible} one of the following occurs, which we have specialized to our present situation.
\begin{enumerate}[(a)]
\item Both $M$ and $M'$ are solid tori and both $K$ and $K'$ are $0$ or $1$--bridge braids (in particular, Berge-Gabai knots \cite{bergesolidtorus, gabai}).
\item $M = S^1 \times D^2 \# L$ and $K$ is the sum of a core of $L$ and a $0$--bridge braid in $S^1 \times D^2$, and $M'$ is a solid torus where $K'$ is the cable on a $0$--bridge braid.
\item $K$ is cabled and the slope of the surgery is the cabling annulus.
\item $M'$ is irreducible and $\bdry$--incompressible.
\end{enumerate}

Case (c) is the desired result, so we rule out the other possibilities.

In case (a), the result of surgery on $K$ is the result of a surgery on the knot that is the core $J$ of $M$.  Since $L-\nbhd(M)$ is irreducible, Boyer-Zhang \cite[Corollary 1.4]{BoyerZhang-CSseminorms} (here, Theorem~\ref{thm:bz}) and Buck-Mauricio \cite[Proposition~5.4]{buckmauricio} tell us that the induced surgery on $J$ must be integral.  However, by Gordon \cite{gordonsatellite}, the induced surgery on $J$ can only be integral if $K$ has winding number $1$. Thus $K$ must be a braid in $M$ with winding number $1$ which implies that $T=\bdry M$ is isotopic to $\bdry X$.  Hence $T$ is not essential.

In case (b), $L-\nbhd(M)$ is the exterior of a non-trivial knot $J$ in $S^3$ where the meridian of $J$ is the compressing slope in $\bdry M$.  Since $Y$ is obtained by surgery on $J$, namely filling the exterior of $J$ with $M'$, the slope of the meridian of $M'$ must be an integral slope for $J$  by Gordon-Luecke \cite{gordonlueckereducible}.  Again, Gordon \cite{gordonsatellite} implies that $K'$ must have winding number $1$ and it similarly follows that $T$ is not essential.

Finally, in case (d), $\bdry M' = T$ must be incompressible in $Y$.  Since $Y$ is non-prime it contains an essential sphere, and $T$ must be isotopic to a torus that is disjoint from any essential sphere in $Y$.  Thus either $M'$ contains an essential sphere or $Y-\nbhd(M) = L-\nbhd(M)$ contains an essential sphere, a contradiction since both of these manifolds are irreducible.
\end{proof}


\section{Acknowledgements}
We would like to thank Mauro Mauricio for the discussions about reducible surgeries on knots in lens spaces that spurred this article.
We also thank Jessica Purcell for a conversation about minimally twisted chain links and the sharpening of our original argument in Lemma~\ref{lem:hypchainlinksurgery}.  This work was partially supported by Simons Foundation grant \#209184 to Kenneth Baker.


	\bibliographystyle{plain}
	\bibliography{reducible}

\end{document}